\documentclass[11pt, a4paper,reqno]{amsart}

\usepackage{amsmath}%
\usepackage{amsfonts}%
\usepackage{amssymb}%
\usepackage{graphicx}
\usepackage{t1enc,epsfig,multicol}
\newtheorem{theorem}{Theorem}
\theoremstyle{plain}

\newtheorem{corollary}{Corollary}

\newtheorem{definition}{Definition}

\newtheorem{remark}{Remark}

\numberwithin{equation}{section}

\def\diag{\mathcal{D}_{n}\left(\R\right)}

\def\mh{M_{n}^{h}(\C)}

\newcommand{\SPPPT}{\mbox{\protect $\mathcal{P}(n)$}}
\newcommand{\ql}{Q_\lambda}
\newcommand{\qml}{Q_{-\lambda}}

\def\R{\mathbb{R}}   %Reales
\def\C{\mathbb{C}}   %Complejos

\begin{document}
\title[Minimal Hermitian matrices]{A characterization of Hermitian matrices with variable diagonal and smallest operator norm}

\author[Andruchow]{Andruchow, E.}
\author[Larotonda]{Larotonda, G.}
\author[Recht]{Recht, L.}
\author[Varela]{Varela, A.}

\address[Andruchow, Esteban; Larotonda, Gabriel and Varela, Alejandro]
{\newline  Instituto de Ciencias, \newline
\indent Universidad Nacional de General Sarmiento \newline
\indent J.\ M.\ Guti\'errez 1150, (1613) Los Polvorines, Argentina
\newline and
\newline  Instituto Argentino de Matem\'atica (IAM) ``Alberto P. Calder\'on'', CONICET, \newline \indent Saavedra 15, 3er piso \newline \indent 
(C1083ACA), Buenos Aires, Argentina.}

\email[Andruchow, Esteban]{eandruch@ungs.edu.ar}%
\email[Larotonda, Gabriel]{glaroton@ungs.edu.ar}%
\email[Varela, Alejandro]{avarela@ungs.edu.ar}%

\address[Recht, L\'azaro]{\newline  Universidad Sim\'{o}n Bol\'{\i}var,\newline%
\indent Apartado 89000, Caracas 1080A, Venezuela}%
\email[Recht, L\'azaro]{recht@usb.ve}%

%\date{}

%\keywords{Keyword one, keyword two etc.}%
%\dedicatory{Dedicated to Professor XY on the occasion of his seventieth birthday.}

\begin{abstract}
We describe properties of a Hermitian square matrix $M\in M_n(\C)$ equivalent to that of having minimal quotient norm in the following sense:
$$
\|M\|\leq \| M+D\|
$$
for all real diagonal matrices $D\in M_n(\C)$ and $\|\ \|$ the operator norm. These matrices are related to some particular positive matrices with their range included in the eigenspaces of the eigenvalues $\pm \|M\|$ of $M$. We show how a constructive method can be used to obtain minimal matrices of any dimension relating this problem with majorization results in $\R^n$.
\end{abstract}

\maketitle

\section{Introduction}

Let $M_n(\C)$ and $\diag$ be the algebras of complex and real diagonal $n\times n$  matrices. We are interested in describing Hermitian matrices $M\in M_n(\C)$ that  verify 
$$
\|M\|\leq \|M+D\|, \text{ for all } D\in\diag
$$
or equivalently
$$
\|M\|=\text{dist}\left(M,\diag \right)
$$
(where $\|\ \|$ denotes the operator norm).
These $M$ will be called {\sl minimal} matrices and appeared in the study of the minimal length curves in the flag manifold $\SPPPT=\mathcal{U}(M_n(\C))/\mathcal{U}(\mathcal{D}_n(\C))$, where $\mathcal{U}(\mathcal{A})$ denotes the unitary matrices of the algebra $\mathcal{A}$. 
Namely, minimal curves in  \( \SPPPT \) are given by action of (the class of)  exponentials of  anti-Hermitian minimal $n\times n$  matrices. To study anti-Hermitian minimal $n\times n$  matrices is (isometrically) equivalent to investigate  the Hermitian minimal $n\times n$ matrices, and  we find them  notationaly simpler to consider.

The following theorem follows ideas in \cite{duranmatarecht}, where this problem was also studied in the context of von Neumann and C$^*$ algebras. The next result was proved in Theorem 3.3 of \cite{ammrv} as stated here. We write it down in its Hermitian form.

\begin{theorem}
 \label{teo matrices minimales}
 A Hermitian matrix $M\in M_n(\C)$ is minimal in the quotient norm with respect to the diagonals if, and only if,
 there exists a positive semidefinite matrix $P\in \mh$ such that,
 \begin{itemize}
 \item \label{teo matrices minimales C1} $P M^2=\lambda^2\,P$, where $||M||=\lambda$.
 \item \label{teo matrices minimales C2} The diagonal elements of the product $P M$ are all zero.
 \end{itemize}
 \end{theorem}

Previous attempts to describe minimal matrices beyond this theorem were done in \cite{ammrv} in $3\times 3$ matrices. In that work, all $3\times 3$ minimal matrices were parametrized. However, Theorem \ref{teo matrices minimales} does not show how to construct $n\times n$ minimal matrices. Our goal in the present paper is to study some properties of $n\times  n$ minimal matrices that allow the construction of them.

This minimal operators were studied recently in \cite{rieffel} where Theorem 2.2 of \cite{ammrv} was used to relate Leibnitz seminorms with quotient norms in C$^*$-algebras.
% 
% The problem of describing this minimal operators in this sense also appeared in the study of best approximation from C$^*$-subalgebras (see \cite{rieffel}). 

\section{Preliminaries and notation}
Let us call with $\mh$ the set of $n\times n$ Hermitian complex matrices and with $\diag$ the subset of the diagonal real matrices.
In these algebras we will denote with $\|\ \|$ the usual operator norm, that is $\| A\|=\text{max}\{|\sigma|:\sigma \text{ is an eigenvalue of } A\}$ if $A\in\mh$.

Given a matrix $A\in\mh$ we will call with $\lambda(A)\subset \R^n$ the set of the eigenvalues of $A$ in decreasing order and counting multiplicity, that is, 
$$
\lambda(A)=\left(\lambda_1, \lambda_2, \dots, \lambda_n\right),
$$ 
with $\lambda_1\geq\lambda_2\geq \dots\geq \lambda_n$, and $\lambda_i$ an eigenvalue of $A$. The spectrum of $A$ will be denoted with 
$$
\sigma(A)=\{\sigma_1,\sigma_2,\dots,\sigma_r\}
$$
where the eigenvalues of $A$ are listed just once and without any prescribed order.

We will denote with $\{e_i\}_{i=1}^n$ the usual canonical basis of $\C^n$ and with ``tr'' the usual trace of matrices. 
% 
% Given a vector $v\in\C^n$ we wil denote with  $\overrightarrow{|v|^2}$ the vector of $\R^n_+$ obtained from $v$ taking the absolute value of each coordinate squared, that is $\overrightarrow{|v|^2}=(|v_{1}|^2,|v_{2}|^2,\dots, |v_{n}|^2 )=\sum_{j=1}^n |v_{j}|^2 e_i$.

 Observe that if $M\in\mh$ and $D\in\diag$ then $(M+D)\in\mh$. Let us consider the quotient $\displaystyle{\mh}/{\diag}$ and the quotient norm 
$$
|||\ [M]\ |||=\min_{D\in \diag} \|M+D\|=\text{dist}\left(M,\diag\right)
$$
for $[M]=\{M+D: D\in \diag\} \in\displaystyle{\mh}/{\diag}$. The minimum is obtained by compactness arguments. 

\begin{definition}
A matrix $M\in\mh$ will be called {\bf minimal for $\diag$} or just {\bf minimal} if 
$$
\|M\|\leq\|M+D\|, \ \ \text{ for all } D\in\diag
$$
or equivalently, if $\| M\|=|||\ [M]\ |||=\underset{D\in \diag}{\min} \|M+D\|=\text{dist}\left(M,\diag\right)$.
\end{definition}

\begin{remark} Observe that if $M$ is a minimal matrix then its spectrum is ``centered'' in the sense that if $\|M\|=\lambda$,
then $-\lambda\in \sigma(M)$.
\end{remark}

For $a_1,a_2, \dots, a_n\in\R$ we will denote with diag$(a_1,a_2, \dots, a_n)$ or with diag$\{a_1,a_2, \dots, a_n\}$ the diagonal matrix of $\diag$ with $a_1,a_2, \dots, a_n$ in the diagonal.

Given $v\in\C^n$, we will call with $v\otimes v$ the linear map from $\C^n$ to $\C^n$ defined by $(v\otimes v)(x)= \langle x, v\rangle v$, for $x\in\C^n$ and $\langle \ ,\ \rangle$ the usual inner product in $\C^n$.

For $M\in \mh$ and $v\in\C^n$ we will write $\overline{M}$ and $\overline{v}$ to denote the matrix and vector
 obtained from $M$ and $v$ by conjugation of its canonical coordinates.

If $M, N\in M_n(\C)$ we will denote with $M\circ N$ the Schur or Hadamard product of those matrices defined by $(M\circ N)_{i,j}=M_{i,j} N_{i,j}$ for $1\leq i,j\leq n$. 
Therefore, if $v\in\C^n$, with coordinates in the canonical basis given by $v=(v_1, v_2,\dots, v_n)$, 
$$
v\circ \overline{v}=(|v_1|^2, |v_2|^2,\dots, |v_n|^2)\in\R_{+}^n.
$$
The usual matrix product will be denoted with $M N$, for $M, N\in M_n(\C)$.

\section{Minimal matrices}
The following is a slight variation of Theorem \ref{teo matrices minimales}.
\begin{theorem}
 \label{teo: matrices conmutan}
A matrix $M\in\mh$ is minimal in the quotient norm with respect to the diagonals if, and only if,
there exists a positive semidefinite matrix $P\in \mh$ such that,
\begin{itemize}
\item \label{teo: matrices conmutan C1} $P M^2=\lambda^2\,P$, where $||M||=\lambda$.
\item \label{teo: matrices conmutan C2} The diagonal elements of the product $P M$ are all zero,
\item \label{teo: matrices conmutan C3} $P$ commutes with $M$.
\end{itemize}
\end{theorem}

\begin{proof}%{Proof of Theorem~\ref{teo matrices minimales}}
Since $M$ is minimal if and only if the first two conditions of Theorem \ref{teo: matrices conmutan} hold for a positive $P$ (see, Theorem \ref{teo matrices minimales}), we only have to prove that a positive matrix $P_0$ that fulfills the three conditions of Theorem \ref{teo: matrices conmutan} can be chosen if $M$ is minimal.

Suppose that the spectrum of $M$ is $\sigma(M)=\{\lambda, -\lambda, \sigma_1,\dots, \sigma_{r}\}$, with $\|M\|=\lambda$ ($\lambda> |\sigma_i|$), for $1\leq i\leq r$ and that $\ql$, $\qml$, $Q_{\sigma_1}$, $\dots$, $Q_{\sigma_r}$ are the corresponding spectral projections of $M$. Then, 
$$
M=\lambda \ql-\lambda \qml+\sum_{i=1}^r \sigma_i Q_{\sigma_i}.
$$
Observe that since $\lambda>|\sigma_i|$, for $1\leq i\leq r$, then the spectral projection of $M^2$ for the eigenvalue $\lambda^2$ is $\ql+\qml$.
% %\text{ and }\
% and
% $$
% M^2=\lambda^2 \left(\ql+\qml\right)+\sum_{i=1}^r \sigma_i^2 Q_{\sigma_i}=\lambda^2 \left(\ql+\qml\right)+\sum_{j=1}^s \tau_j q_{\tau_j}.
% $$
% where $\tau_j$ and $q_{\tau_j}$ are the eigenvalues and corresponding spectral projections of $M^2$ (with $\lambda^2>\tau_j$).
% %(with eventually $\tau_j=\sigma_{i_1}^2+\sigma_{i_2}^2$ and $q_{\tau_j}=Q_{\sigma_{i_1}}+Q_{\sigma_{i_2}}$).

Since we are supposing that $M$ is minimal, there exists a positive semidefinite matrix $P$ that verifies the two conditons of Theorem \ref{teo matrices minimales}. Then, since $P M^2=\lambda^2 P$, then $P$ commutes with $M^2$. 
Then taking the same unitary to diagonalize $P$ and $M^2$, and using that $P M^2=\lambda^2 P$, it can be proved that $P Q=0$ for every spectral projection $Q$ of $M^2$, except the one corresponding to the eigenvalue $\lambda^2$, that is, $\ql+\qml$. Therefore, the representation of $P$ and $M$ in blocks corresponding with the orthogonal decomposition given by the range of the orthogonal projections $\ql$, $\qml$ and $I-\ql-\qml$ (respectively) is
$$
P=
\left(
\begin{array}{ccc}
P_{1,1} & P_{1,2}& 0  \\
P_{1,2}^* & P_{2,2}& 0 \\
0 & 0 &0 
\end{array}
\right)
%\begin{array}{c} \ql \\ \qml\\ I-\ql-\qml \end{array}
\text{ and }
M=
\left(
\begin{array}{ccc}
\lambda  & 0& 0  \\
0 & -\lambda & 0 \\
0 & 0 &\sum_{i=1}^r \sigma_i Q_{\sigma_i} 
\end{array}
\right).
%\begin{array}{c} \ql \\ \qml\\ I-\ql-\qml \end{array}
$$
Then, using the second condition of Theorem \ref{teo matrices minimales}, that is, $\langle PM e_i,e_i\rangle=0$ for the canonical basis $\{e_i\}_{i=1,\dots,n}$, we obtain that
$$
\langle PM e_i,e_i\rangle=
\langle
\left(
\begin{array}{ccc}
\lambda P_{1,1} & -\lambda P_{1,2}& 0  \\
\lambda P_{1,2}^* &-\lambda P_{2,2}& 0 \\
0 & 0 &0 
\end{array}
\right)
\left(
\begin{array}{c}
\ql e_i \\
\qml e_i\\
e_i-\ql e_i-\qml e_i
\end{array}
\right)
,\left(
\begin{array}{c}
\ql e_i \\
\qml e_i\\
e_i-\ql e_i-\qml e_i
\end{array}
\right)
\rangle
=0
$$
for all $i=1,\dots,n$.
Then, since $P_{1,1} \ql=P_{1,1}$, $P_{1,2} \qml=P_{1,2}$, $P_{1,2}^* \ql=P_{1,2}^*$ and $P_{2,2} \qml=P_{2,2}$, it follows that
$$
\langle \lambda P_{1,1}e_i-\lambda P_{1,2} e_i, e_i\rangle+ \langle \lambda P_{1,2}^*e_i-\lambda P_{2,2}e_i, e_i\rangle=\lambda \langle (P_{1,1}-P_{2,2}) e_i, e_i\rangle+ \lambda \langle (P_{1,2}^*-P_{1,2})e_i, e_i\rangle=0
$$
for all $i=1,\dots,n$. The term $\langle(P_{1,1}-P_{2,2}) e_i, e_i\rangle$ in the previous equation is real, since $P_{1,1}=\ql P\ql$ and $P_{2,2}=\qml P\qml$ are positive semidefinite matrices. The term $\langle (P_{1,2}^*-P_{1,2})e_i, e_i\rangle$ is purely imaginary since $\overline{\langle (P_{1,2}^*-P_{1,2})e_i, e_i\rangle}=-\langle (P_{1,2}^*-P_{1,2})e_i, e_i\rangle$. Then both terms must be zero, which implies that $\langle P_{1,1} e_i,e_i\rangle=\langle P_{2,2} e_i, e_i\rangle$. Therefore, the matrices $P_{1,1}$ and $P_{2,2}$ have the same diagonal in the canonical basis $\{e_i\}_{i=1,\dots,n}$.
Then, if we define 
$$
P_0=
\left(
\begin{array}{ccc}
P_{1,1} &  0& 0  \\
0 & P_{2,2}& 0 \\
0 & 0 &0 
\end{array}
\right),
$$
this matrix verifies
\begin{equation}\label{diagonales nulas}
\langle P_0 M e_i, e_i\rangle=
\lambda \left(\langle P_{1,1} e_i,e_i\rangle-\langle P_{2,2} e_i, e_i\rangle\right)=0 
\end{equation}

Moreover, $P_0\geq 0$ and, using the block decompositions of $M$ and $P_0$, it also verifies that
\begin{equation}\label{p0m2=l2p0 y conmutan}
P_0 M^2=\lambda^2 P_0 \text{ , and } P_0M=MP_0.
\end{equation}
Therefore, the equalities (\ref{diagonales nulas}) and (\ref{p0m2=l2p0 y conmutan}) imply that the positive semidefinite matrix $P_0$ verifies the three properties required.
\end{proof}

\begin{remark}
 Observe that the matrix $P_0$ of Theorem \ref{teo: matrices conmutan} was obtained as a diagonal block matrix in terms of the spectral projections $\ql$, $\qml$, $I-\ql-\qml$ of $M$ from any matrix $P$ verifying Theorem \ref{teo matrices minimales}.
\end{remark}

The proof of Theorem \ref{teo: matrices conmutan} suggests another equivalent condition for being minimal:

\begin{corollary}\label{corolario: p mas y p menos}
 Given a matrix $M\in\mh$  the following statements are equivalent:
\begin{itemize}
\item $M$ is minimal
\item  $\pm\|M\|\in\sigma(M)$ and 
there exist a pair positive semidefinite matrices $P_+, P_-\in \mh$, such that, if $Q_{\|M\|}$, $Q_{-\|M\|}$ are the spectral projections of $M$ with respect to the eigenvalues $\pm \|M\|$ respectively, they satisfy the following
\begin{itemize}
 \item[i) ] $P_+ Q_{\|M\|}=Q_{\|M\|} P_+=P_+$
\item[ii) ] $P_- Q_{-\|M\|} =Q_{-\|M\|} P_-=P_-$
\item[iii) ] $\langle P_- e_i, e_i\rangle=\langle P_+ e_i, e_i\rangle$, for all $e_i$, $i=1,\dots, n$, the canonical basis of $\C^n$.
\end{itemize}

\end{itemize}

 \end{corollary}
\begin{proof}
 If we suppose that $M$ is minimal it suffices to choose $P_+=P_{1,1}$ and $P_-=P_{2,2}$ from the proof of Theorem \ref{teo: matrices conmutan}. 

If there exist such $P_+$ and $P_-$ then a direct calculation shows that the matrix $P=P_+ +P_-$ fulfills the requirements of Theorem \ref{teo: matrices conmutan}, and therefore $M$ is minimal.
\end{proof}

This corollary motivates the following definition.

\begin{definition}
 Given a positive semidefinite matrix $P\in\mh$, another positive semidefinite $Q\in\mh$ is called a {\bf companion} matrix of $P$ if,  $P Q=0$ (being $0$ the null matrix) and they both have the same diagonal in the canonical basis. We will say that $P$ has a companion $Q$, or that $P$ and $Q$ are companions.
\end{definition}

\begin{remark}

i) Note that if $P$ is a companion of $Q$, then $Q$ is a companion of $P$.

ii) If $P$ and $Q$ are companions then they must have the same trace since they have the same diagonal.

iii) If $P$ is a companion of $Q$ and $P\neq 0$, then $Q\neq 0$. This holds because if $Q=0$ then the diagonal of $P$ must be zero in the canonical basis. This yields to $P=0$ since $P$ is positive semidefinite, a contradiction. Therefore, if $P$ and $Q$ are companions and one of them is $0$, then the other must be $0$.

iv) Observe that not every positive semidefinite matrix $P$ has a companion. For example, if $P$ is invertible, then it has not got any companion matrix. Therefore, if a matrix $P$ has a companion, then $P$ must have non trivial kernel. 

v)
 Note that a matrix $P$ could have many companions. Take por example any $3\times 3$ complex Hadamard matrix $H$ (that is a matrix such that $|H_{i,j}|=1$ with orthogonal rows and columns), and consider the unitary matrix $U=\frac1{\sqrt{3}} H$. Then, if $\text{diag}\left(a,b,c\right)$ denotes the $3\times 3$ diagonal matrix with $a$, $b$ and $c$ in its diagonal, and we define $P=U \text{diag}\left(4,0,0\right)  U^*$ and $Q_t=U \text{diag}\left(0,4-t,t\right)  U^*$ for $t\in\R$ and $0\leq t\leq 4$, an easy check proves that $\{Q_t\}_{0\leq t\leq 4}$ are all different companion matrices of $P$.

\end{remark}

In the following corollary, if $Q\in M_n(\C)$, then ran$(Q)$ will denote the range of the corresponding linear transformation.

\begin{corollary}\label{corolario relacion entre companneras y minimales}
Given $S_1$, $S_2$ subspaces of $\C^n$ with $S_1\perp S_2$, then the following statements are equivalent:
\begin{itemize}
 \item[i) ] There exist positive semidefinite matrices $P_1, P_2\in\mh$,  with ran$(P_1)\subset S_1$ and ran$(P_2)\subset S_2$, such that $P_1$ and $P_2$ are companions.
\item [ii) ] $M=\lambda\ P_{S_1}-\lambda\ P_{S_2}+R$ is a minimal matrix, for every $\lambda>0$ and $R\in\mh$  such that $P_{S_1} R=P_{S_2} R=0$ and $\|R\|<\lambda$ (with $P_{S_1}$ and $P_{S_2}$ the respective orthogonal projections onto the subspaces $S_1$ and $S_2$).
\end{itemize}
\end{corollary}

\begin{proof}
 Let us suppose first that $P_1$ and $P_2$ are companion matrices with the hypothesis of i). Consider then $\lambda>0$ and a matrix $M=\lambda\ P_{\text{ran}(P_1)}-\lambda\ P_{\text{ran}(P_2)}+R$, with $R$ such that its range is orthogonal to that of $P_1$ and $P_2$ and $\|R\|<\lambda$. Then taking $P=P_1+P_2$ it is easy to verify that $P$ and $M$ satisfy the conditions of Theorem \ref{teo matrices minimales} that imply that $M$ is minimal with $\|M\|=\lambda$.

Let us suppose now that $M=\lambda\ P_{S_2}-\lambda\ P_{S_2}+R$ as in item ii) is a minimal matrix. 
Then using that $S_1\perp S_2$, that ran$(R)$ is orthogonal to $S_1\oplus S_2$ and that $\|R\|<\lambda$, it is apparent that the spectral projections $\ql$, $\qml$ of $M$ with respect to the eigenvalues $\lambda$ and $-\lambda$ verify that $\ql=P_{S_1}$ and $\qml=P_{S_2}$.
Then there exists a positive semidefinite $P\in\mh$ that verifies the three statements of Theorem \ref{teo: matrices conmutan}. Therefore $P$ commutes with $M$.
As in the proof of Theorem \ref{teo: matrices conmutan} it can be proved that the representation of $P$ as a block matrix with respect to the orthogonal subspaces $S_1$, $S_2$ and $(S_1\oplus S_2)^\perp$ is
$$
P=
\left(
\begin{array}{ccc}
P_{1,1} &  0& 0  \\
0 & P_{2,2}& 0 \\
0 & 0 &0 
\end{array}
\right).
$$
We shall prove that $P_1=P_{1,1}=P_{S_1} P P_{S_1}$ and $P_2=P_{2,2}=P_{S_2} P P_{S_2}$ fulfill the conditions of i). Since $P$ is positive semidefinite it is apparent that $P_1$ and $P_2$ are also positive semidefinite. By definition, ran$(P_1)\subset S_1$ and ran$(P_2)\subset S_2$ and $P_1 P_2=0$. 

Moreover, since $PM$ has zero diagonal in the canonical basis, then 
$$
PM=
\left(
\begin{array}{ccc}
\lambda P_{1} &  0& 0  \\
0 & -\lambda P_{2}& 0 \\
0 & 0 &0 
\end{array}
\right)
$$
has zero diagonal in the canonical basis of $\C^n$. That means that $\lambda \langle (P_1-P_2) e_i,e_i\rangle=0$ for the canonical basis $\{e_i\}_{i=1,\dots,n}$ of $\C^n$, and then the diagonals of $P_1$ and $P_2$ coincide in that basis. Therefore, $P_1$ is a companion of $P_2$.
% 
% Then $P$ commutes with the spectral projections $P_{S_1}$ and $P_{S_2}$ of $M$. Let us name $P_1=P_{S_1} P  P_{S_1}$ and $P_2=P_{S_2} P P_{S_2}$. Moreover, since $PM^2=\lambda^2 P$, then as in the proof of Theorem \ref{teo: matrices conmutan} it can be proved that $PR=0$.
% 
%  $P P_{S_1}$ 
\end{proof}

\section{Characterization of companion matrices}

Corollary \ref{corolario relacion entre companneras y minimales} gives a direct relation between minimal matrices and pairs of companion matrices. Moreover, if one has a pair of companion matrices then a minimal matrix can be constructed as in ii) of that corollary. In this section we will describe some of the properties of the companion matrices.

Recall that, as it was mentioned in the preliminaries, for a given vector $v\in\C^n$, 
$$
v\circ \overline{v}=(|v_{1}|^2,|v_{2}|^2,\dots, |v_{n}|^2 )=\sum_{j=1}^n |v_{j}|^2 e_i \in\R^n_+,
$$ 
if $v$ has canonical coordinates $(v_{1}, v_{2},\dots, v_{n})$. 
For given vectors $\{w_k\}_{k=1}^m\subset \C^n$ we will denote with $\text{K}\left(\{w_k\}_{k=1}^{m}\right)$ and $\text{co}\left(\{w_k\}_{k=1}^{m}\right)$ 
%and $<\{w_k\}_{k=1,\dots,m}>$ the subspace
the cone and the convex hull generated by them (respectively).

\begin{theorem}\label{teorema: caracterizacion companneras}
Let $P\in\mh$ be a positive semidefinite matrix, its eigenvalues counted with multiplicity given by $\lambda(P)=\left(a_1, a_2,\dots, a_r,0,\dots,0\right)$, with $a_i>0$, $1\leq r < n$. Then the following properties of $P$ are equivalent
\begin{itemize}
 \item[i) ] $P$ has a companion $Q$.
\item[ii) ] There exist a set of orthonormal eigenvectors $\{v_1, v_2, \dots, v_r\}$ corresponding to the (strictly) positive eigenvalues $a_1, a_2, \dots a_r$ of $P$ and another set of orthonormal  eigenvectors  $\{v_{r+1}, v_{r+2}, \dots, v_n\}$ of the kernel of $P$, and $x_j\geq 0$ such that 
\begin{equation}
\label{ecuacion autovalores y autovectores companneras}
\sum_{i=1}^r a_i (v_i\circ\overline{v_i})=\sum_{j=r+1}^n x_j (v_j\circ\overline{v_j}).
\end{equation}

% and $\overrightarrow{|v_k|^2}
% =(|v_{1,k}|^2,|v_{2,k}|^2,\dots, |v_{n,k}|^2 )=\sum_{j=1}^n |v_{j,k}|^2 e_i$, if the eigenvector $v_k$ has canonical coordinates $(v_{1,k}, v_{2,k},\dots, v_{n,k})$. 
\item[iii) ]  There exist a set of orthonormal eigenvectors $\{v_1, v_2, \dots, v_r\}$ corresponding to the (strictly) positive eigenvalues $a_1, a_2, \dots a_r$ of $P$ and another set of orthonormal eigenvectors  $\{v_{r+1}, v_{r+2}, \dots, v_n\}$ of the kernel of $P$ such that 
$$
\sum_{i=1}^r a_i (v_i\circ\overline{v_i})\in \text{K}\left(\left\{ v_j\circ\overline{v_j}\right\}_{j=r+1}^{n} \right).
$$
% 
% to the cone generated by with $\overrightarrow{|v_k|^2}=(|v_{1,k}|^2,|v_{2,k}|^2,\dots, |v_{n,k}|^2 )$ if the eigenvector $v_k$ has canonical coordinates $v_{1,k}, v_{2,k},\dots, v_{n,k}$. 
\item [iv) ] 
There exists a set of orthonormal eigenvectors $\{v_i\}_{i=1}^{r}$ of $P$ corresponding to the (strictly) positive eigenvalues $a_1, a_2, \dots a_r$ of $P$ and  orthogonal eigenvectors  $\{v_j\}_{j=r+1}^{r+s}\subset\text{Ker}(P)$,  that verify
$$
\sum_{i=1}^{r} \frac{a_i}{\text{tr}(P)}  v_i\circ\overline{v_i}\in \text{co}\left(\left\{ v_j\circ\overline{v_j}\right\}_{j=r+1}^{r+s}\right).
$$
% \item [v) ] 
% There exists a set of orthonormal eigenvectors $\{v_i\}_{i=1}^{r}$ of $P$ corresponding to the (strictly) positive eigenvalues $a_1, a_2, \dots a_r$ of $P$ and  orthogonal eigenvectors  $\{v_j\}_{j=r+1}^{r+s}\subset\text{Ker}(P)$,  that verify
% $$
% \text{co}\left(\left\{ v_i\circ\overline{v_i}\right\}_{i=1}^{r}\right)
% \cap
% \text{co}\left(\left\{ v_j\circ\overline{v_j}\right\}_{j=r+1}^{r+s}\right)\neq \emptyset.
% $$
\end{itemize}
\end{theorem}

\begin{proof}
Let us suppose first that $P$ has a companion $Q$, and the spectrum of $P$, counting multiplicity of eigenvalues and in descending order, is $\lambda(P)=\left(a_1, a_2, \dots,  a_{r}, 0,\dots, 0\right)$, with $a_r>0$. Then, since $P Q=0$, they commute, and therefore we can choose a unitary matrix $V$ that diagonalizes both $P$ and $Q$. 
We can also choose $V$ in the following way:
\begin{equation}\label{unitaria autovectores columna}
V=
\left(
\begin{array}{cccc}
v_{1,1} & v_{1,2}&\dots & v_{1,n}\\
v_{2,1} & v_{2,2}&\dots & v_{2,n}\\
\vdots & \vdots &\ddots & \vdots\\
v_{n,1} & v_{n,2}&\dots & v_{n,n}\\
\end{array}
\right)
\end{equation}
where the columns are the coordinates in the canonical basis of $\C^n$ of an orthonormal basis $\{v_i\}_{1\leq i\leq n}$ of eigenvectors of $P$ and $v_i=(v_{1,i} , v_{2,i} , \dots , v_{n,i})$ is the corresponding eigenvector of $a_i$ (for $1\leq i\leq r$). 
Then, this $V$ verifies that $P=V D_P V^*$ and $Q=V D_Q V^*$, where $D_P$ is a diagonal matrix with $\lambda(P)$ in its diagonal and $D_Q$ is a diagonal with the eigenvalues of $Q$ in its diagonal. Since $Q$ must be positive and $P Q=0$, then the diagonal of $D_Q$ has to be of the form $\{0,0,\dots,0,x_{r+1}, x_{r+2},\dots, x_n\}$ with $x_i\geq 0$, for $r+1\leq i\leq n$. Moreover, since $P$ and $Q$ have identical diagonals in the canonical basis, then considering the decompositions
\begin{align*}
P=&VD_PV^*= \\
=&\left(
\begin{array}{cccc}
v_{1,1} & v_{1,2}&\dots & v_{1,n}\\
v_{2,1} & v_{2,2}&\dots & v_{2,n}\\
\vdots & \vdots &\ddots & \vdots\\
v_{n,1} & v_{n,2}&\dots & v_{n,n}\\
\end{array}
\right)
.
\left(
\begin{array}{ccccccc}
a_1 & 0&\dots &\dots& \dots& \dots& 0\\
0 & a_2&0 & \dots& \dots& \dots& 0\\
\vdots & \vdots&\ddots & \vdots& \vdots& \vdots& 0\\
0 & \dots& 0 &a_r& 0 & \dots& 0\\
0 & \dots&0 & 0& 0& \dots& 0\\
\vdots & \vdots &\vdots & \vdots& \vdots &\ddots &\vdots\\
0 & \dots&\dots & \dots& \dots& \dots& 0\\
\end{array}
\right)
.
\left(
\begin{array}{cccc}
\overline{v_{1,1}} & \overline{v_{2,1}}&\dots & \overline{v_{n,1}}\\
\overline{v_{1,2}} & \overline{v_{2,2}}&\dots & \overline{v_{n,2}}\\
\vdots & \vdots &\ddots & \vdots\\
\overline{v_{1,n}} & \overline{v_{2,n}}&\dots & \overline{v_{n,n}}\\
\end{array}
\right)
\end{align*}
and
\begin{align*}
Q=&VD_QV^*= \\
=&\left(
\begin{array}{cccc}
v_{1,1} & v_{1,2}&\dots & v_{1,n}\\
v_{2,1} & v_{2,2}&\dots & v_{2,n}\\
\vdots & \vdots &\ddots & \vdots\\
v_{n,1} & v_{n,2}&\dots & v_{n,n}\\
\end{array}
\right)
.
\left(
\begin{array}{ccccccc}
0 & 0&\dots& \dots& \dots& \dots& 0\\
0 & 0&0 & \dots& \dots& \dots& 0\\
\vdots & \vdots&\ddots & \vdots& \vdots& \vdots& 0\\
0 & \dots& 0 &0& 0 & \dots& 0\\
0 & \dots&0 & 0& x_{r+1}& \dots& 0\\
\vdots & \vdots &\vdots & \vdots& \vdots &\ddots &\vdots\\
0 & \dots&\dots & \dots& \dots& \dots& x_n\\
\end{array}
\right)
.
\left(
\begin{array}{cccc}
\overline{v_{1,1}} & \overline{v_{2,1}}&\dots & \overline{v_{n,1}}\\
\overline{v_{1,2}} & \overline{v_{2,2}}&\dots & \overline{v_{n,2}}\\
\vdots & \vdots &\ddots & \vdots\\
\overline{v_{1,n}} & \overline{v_{2,n}}&\dots & \overline{v_{n,n}}\\
\end{array}
\right)
\end{align*}
we obtain the $n$ following equations
$$
\left\{
\begin{array}{ccc}
 \sum_{i=1}^r a_i |v_{1,i}|^2&=& \sum_{j=r+1}^n x_j |v_{1,j}|^2\\
\sum_{i=1}^r a_i |v_{2,i}|^2&=& \sum_{j=r+1}^n x_j |v_{2,j}|^2\\
\vdots&\vdots&\vdots\\
\sum_{i=1}^r a_i |v_{n,i}|^2&=& \sum_{j=r+1}^n x_j |v_{n,j}|^2
\end{array}
\right..
$$
Then,
$$
\left(\sum_{i=1}^r a_i |v_{1,i}|^2, \sum_{i=1}^r a_i |v_{2,i}|^2, \dots,\sum_{i=1}^r a_i |v_{n,i}|^2\right)
=\left(\sum_{j=r+1}^n x_j |v_{1,j}|^2, \sum_{j=r+1}^n x_j |v_{2,j}|^2, \dots, \sum_{j=r+1}^n x_j |v_{n,j}|^2 \right) 
$$
% 
% 
% \begin{align}
% (\sum_{i=1}^r a_i |v_{1,i}|^2, \sum_{i=1}^r a_i |v_{2,i}|^2,& \dots,\sum_{i=1}^r a_i |v_{n,i}|^2)=\nonumber\\
% =(\sum_{j=r+1}^n x_j |v_{1,j}|^2, \sum_{j=r+1}^n x_j |v_{2,j}|^2,& \dots, \sum_{j=r+1}^n x_j |v_{n,j}|^2 )
% \end{align}
% 
% 
% \begin{align*}
% (\sum_{i=1}^r a_i |v_{1,i}|^2, \sum_{i=1}^r a_i |v_{2,i}|^2,& \dots,\sum_{i=1}^r a_i |v_{n,i}|^2)=\\
% =(\sum_{j=r+1}^n x_j |v_{1,j}|^2, \sum_{j=r+1}^n x_j |v_{2,j}|^2,& \dots, \sum_{j=r+1}^n x_j |v_{n,j}|^2 )
% \end{align*}
and
\begin{equation}\label{cono}
\sum_{i=1}^r a_j \left(|v_{1,i}|^2,|v_{2,i}|^2,\dots,|v_{n,i}|^2\right)= \sum_{j=r+1}^n x_j \left( |v_{1,j}|^2,|v_{2,j}|^2,\dots
,|v_{n,j}|^2\right)
\end{equation}
which proves ii).

Now suppose that ii) holds. If we define $Q=\sum_{j=r+1}^n x_j\ (v_j\otimes v_j)$, with $x_j$ and $v_j$ as in ii), then it verifies that $PQ=0$. Moreover, since the equality (\ref{ecuacion autovalores y autovectores companneras}) is equivalent to the equality of the diagonals of $P$ and $Q$, then $Q$ is a companion of $P$. 

Assertion iii) is equivalent to ii) since 
$
\sum_{j=r+1}^n x_j (v_j\circ\overline{v_j})
$
is a generic element of the cone generated by $\{v_j\circ\overline{v_j}\}_{j=r+1}^{n}$.

Statement ii) implies iv) because the equality (\ref{ecuacion autovalores y autovectores companneras}) is equivalent to the fact that $P$ has the same diagonal than $Q=\sum_{j=r+1}^{r+s} x_j\ (v_j\otimes v_j)$. Then $P$ and $Q$ have the same trace equal to $\sum_{i=1}^{r} a_i=\sum_{j=r+1}^{r+s} x_j$. Therefore  
$$
\sum_{j=1}^r \frac{a_j}{\sum_{i=1}^{r} a_i} v_j\circ\overline{v_j}=
\sum_{j=r+1}^n \frac{x_j}{\sum_{j=r+1}^{r+s} x_j} v_j\circ\overline{v_j}\ \in \
\text{co}\left(\{v_j\circ\overline{v_j}\}_{j=r+1}^{r+s}\right).
$$
If iv) holds then obviously iii) an ii) hold.

% And finally, iv) and v) are obviously equivalent since $\sum_{i=1}^{r} \frac{a_i}{\text{tr}(P)}  v_i\circ\overline{v_i}\in \text{co}\left(\left\{ v_i\circ\overline{v_i}\right\}_{i=1}^{r}\right)$.
\end{proof}

 Considering the results obtained in Corolllary \ref{corolario relacion entre companneras y minimales} and Theorem \ref{teorema: caracterizacion companneras} we can conclude that a matrix $M=\lambda P_{S_1}-\lambda P_{S_2}+R\in\mh$ (with $S_1\perp S_2$ and $R\in\mh$ with $\|R\|<\lambda$) is minimal, if and only if, there exist orthonormal vectors $\{v_i\}_{i=1}^r\subset S_1$ and $\{v_j\}_{j=r+1}^{r+s}\subset S_2$ such that
$$
\text{co}\big( \{v_i\circ\overline{v_i}\}_{i=1}^{r} \big) \cap
\text{co}\left(\{v_j\circ\overline{v_j}\}_{j=r+1}^{r+s}\right) \neq \emptyset.
$$

Note also that any minimal matrix is necessarily of this form.

Moreover, given a matrix $M\in\mh$, then $M$ is minimal, if and only if, there exists a unitary matrix $U$ such that $U^*MU=\hbox{diag}\left(\lambda(M)\right)$ and the rows of the unistochastic matrix $U^*\circ \overline{U^*}$ have the required properties with respect to the eigenspaces of $\lambda=\|M\|$ and $-\lambda$ of $M$. Namely, that 
$$
\text{co}\left(\{v_i\circ\overline{v_i}\}_{i=1}^{r}\right)\cap\text{co}\left(\{v_j\circ\overline{v_j}\}_{j=r+1}^{r+s}\right)\neq \emptyset,
$$
where $\{v_i\}_{i=1}^{r}$ are the corresponding orthogonal eigenvectors of $\lambda$ (and rows of $U$) and $\{v_j\}_{j=r+1}^{r+s}$ are the corresponding orthogonal eigenvectors of $-\lambda$ (and rows of $U$).

Observe that following the notation of Theorem \ref{teorema: caracterizacion companneras} ii), since $\sum_{i=1}^r a_i=\sum_{j=r+1}^n x_j$, then 
$$
(0,0,\dots, 0)\prec (a_1,\dots,a_r,-x_{r+1},\dots,-x_n)=\vec{ax}
$$
 (where $\prec$ is the usual notation for majorization of vectors in $\R^n$, see \cite{marshall olkin}). Then the equations in (\ref{cono}) prove that the matrix
$$
V^*\circ \overline{V^*}=
\left(
\begin{array}{cccc}
|v_{1,1}|^2 & |v_{2,1}|^2 &\dots & |v_{n,1}|^2\\
|v_{1,2}|^2 & |v_{2,2}|^2&\dots & |v_{n,2}|^2\\
\vdots & \vdots &\ddots & \vdots\\
|v_{1,n}|^2 & |v_{2,n}|^2&\dots & |v_{n,n}|^2\\
\end{array}
\right)
$$
obtained from the matrix (\ref{unitaria autovectores columna}) is a doubly stochastic (in fact, unistochastic) matrix that verifies $(0,\dots,0)=\vec{ax}.\left(V^*\circ \overline{V^*}\right)$.
This suggests a relation with results in majorization of vectors in $\R^n$.

Take any $n$-tuple  $\vec{a0x}=(a_1,\dots, a_r, 0, \dots, 0, -x_1,\dots, -x_s)\in\R^n$, with $a_i\geq 0$ and $x_j\geq 0$, such that  $\sum_{i=1}^r a_i=\sum_{j=1}^s x_j$. Then it is apparent that $(0,\dots,0)\prec \vec{a0x}$. 
Therefore a concrete unitary or orthogonal matrix $U$ can be found (see \cite{horn, mirsky}) such that $(0,\dots,0)=\vec{a0x}.(U\circ \overline{U})$. Then, if we call with $v_k$ the $k$-th column of $U^*$ (for $k=1,\dots,n$), any matrix of the form 
\begin{equation}\label{ecuacion matrices minimales}
M=\lambda \sum_{i=1}^r v_i\otimes v_i+\sum_{h=r+1}^{n-s} \lambda_h (v_h\otimes v_h)-\lambda \sum_{j=n-s+1}^{n} v_j\otimes v_j
\end{equation}
is minimal provided that $\lambda>0$, $\lambda_h\in\R$ and $|\lambda_h|<\lambda$.
 These results, together with Corollary \ref{corolario relacion entre companneras y minimales} and Theorem \ref{teorema: caracterizacion companneras}
 allow to construct minimal matrices of any size.

The method to obtain minimal matrices $M$ mentioned in (\ref{ecuacion matrices minimales})  relies on which is the unitary $U$ retrieved from the unistochastic matrix.
The work of \cite{dhillon heath sustik tropp} shows different algorithms to find such a unitary or even orthogonal matrix $U$ that verifies $\overrightarrow{0}=\vec{a0x}. U\circ \overline{U}$. Nevertheless, the set of all posible unitaries $U$ that give the same unistochastic matrix is not known in general.
 The works of \cite{tadej zyczkowski} and \cite{zyczkowski kus slomczynski sommers} study the problem of describing the different matrices $U$ such that the mapping $U \mapsto U\circ \overline{U}$ gives the same unistochastic matrix.  
% 
% In general, for a given unistochastic matrix $U_0\circ \overline{U_0}$ all the unitary matrices $U$ that verify $U\circ \overline{U}=U_0\circ \overline{U_0}$ are not known. Therefore, since the process that we show here to construct minimal matrices is related to the problem of finding those unitaries, then finding all the concrete minimal matrices.
% Nevertheless, since there exist 

%There are also many papers related to the problem of constructing a matrix with prescribed diagonal and eigenvalues (see \cite{chu}, \cite{davies higham}, \cite{leite richa tomei}, \cite{zha zhang}). The works of \cite{haagerup} and \cite{beauchamp nicoara} study this construction when the doubly stochastic matrix is a Hadamard complex matrix. 

\begin{remark}
In \cite{ammrv} a different characterization of minimal $3\times3$ matrices were given. It was shown that given a $3\times 3$ matrix $M$, with $\lambda(M)=\left(\lambda,\mu,-\lambda\right)$, $|\mu|\leq\lambda=\|M\|$, then, $M$ was minimal, if and ony if, there exists an orthonormal eigenvector $v_\lambda$ of the eigenvalue $\lambda$ and a orthonormal eigenvector $v_{-\lambda}$ of the eigenvalue $-\lambda$ such that
 $v_\lambda\circ\overline{v_\lambda}=v_{-\lambda}\circ\overline{v_{-\lambda}}$. The statement remains valid if any of the eigenvalues has multiplicity two ($\mu=\pm \lambda$).
 The following is an example of a $4\times 4$ minimal Hermitian matrix where this condition does not hold. 
Let
$$
M=
\left(
\begin{array}{cccc}
 \frac{9}{14} & -\frac{15}{14}-\frac{i}{7} &
   -\frac{1}{7}+\frac{5 i}{7} & \frac{2}{7}+\frac{6
   i}{7} \\
 -\frac{15}{14}+\frac{i}{7} & \frac{13}{14} &
   -\frac{1}{7}+i & \frac{6 i}{7} \\
 -\frac{1}{7}-\frac{5 i}{7} & -\frac{1}{7}-i &
   \frac{5}{7} & -1-\frac{2 i}{7} \\
 \frac{2}{7}-\frac{6 i}{7} & -\frac{6 i}{7} &
   -1+\frac{2 i}{7} & \frac{5}{7}
\end{array}
\right).
$$
 Then $\lambda(M)=\left(2, 2, 1, -2\right)$, and the eigenspace of the eigenvalue $2$ is generated by the orthonormal eigenvectors
$
v_1=\frac1{5 \sqrt{2}}\left(-1-2 i,5,-3-i,1-3 i\right)
$
and 
$
v_2=\frac1{10 \sqrt{14}}\left(17-11 i,-15+5 i,-9+17 i,3-19 i\right).
$
The vector $w=\frac1{2\sqrt{2}} \left(1-i,1-i,1+i,1+i\right)$ is a orthonormal eigenvector of eigenvalue $-2$.
 A direct calculation shows that for $\alpha=\frac29$, then $\alpha (v_1\circ \overline{v_2})+(1-\alpha) (v_2\circ\overline{v_2})=w\circ \overline{w}=(\frac14,\frac14,\frac14,\frac14)$, which is enough to prove that $M$ is minimal (using Theorem \ref{teorema: caracterizacion companneras} and  Corollary \ref{corolario relacion entre companneras y minimales}). Nevertheless, there is not a single eigenvector $v$ in the eigenspace of $\lambda$ such that $v\circ \overline{v}=w\circ\overline{w}$. This follows after writing $v=\beta v_1+\gamma v_2$ with $\beta, \gamma\in\C$, and $|\beta|^2+|\gamma|^2=1$, and proving that $v\circ \overline{v}=w\circ\overline{w}$ could never happen (note that we can suppose that $\gamma=\sqrt{1-|\beta|^2}$).

\end{remark}

% 
% \item[iii) ] item is this: choose any sequence of numbers $a_1\geq a_2\geq \dots\geq a_r>0$ and $0<x_1\leq x_2\leq\dots\leq x_s$ such that $\sum_{i=1}^r a_i =\sum_{j=1}^s x_j$. 
% Then $\overrightarrow{0}=(0,0,\dots,0)\prec a_1, a_2, \dots, a_r,, 0,\dots, 0,-x_1, -x_2,\dots, -x_s)=\overrightarrow{ax0}\in\R^n$. Then find an orthostochastic matrix $\mathcal{O}\in M_n(\R)$ such that $\overrightarrow{0}=\overrightarrow{ax0}.\mathcal{O}$. Then find any unitary $U\in M_n(\C)$ such that $\mathcal{O}=U\circ \overline{U}$. 

% %%%%%%%((((((((((((())))))))))))))%%%%%%%%%

\vskip1cm


\begin{thebibliography}{XXXXXX}


\bibitem{ammrv}
Andruchow, Esteban; Mata-Lorenzo, Luis E.; Mendoza, Alberto; Recht, L\'azaro; Varela, Alejandro. {\it Minimal matrices and the corresponding minimal curves on flag manifolds in low dimension}. Linear Algebra Appl. 430 (2009), no. 8-9, 1906-1928. 

% 
% \bibitem{beauchamp nicoara}
% Beauchamp, Kyle; Nicoara, Remus. Orthogonal maximal abelian $\ast$-subalgebras of the $6\times6$ matrices. Linear Algebra Appl. 428 (2008), no. 8-9, 1833--1853.

% \bibitem{carlen lieb} Carlen, Eric A.; Lieb, Elliott H.  {\it Short proofs of theorems of Horn and Mirsky on diagonals and eigenvalues of matrices}, Electron. J. Linear Algebra 18 (2009), 438--441.
% arXiv:0904.0734v2 [math.FA]
% 
% \bibitem{chu} M. T. Chu, Constructing a Hermitian matrix from its diagonal entries and eigenvalues, SIAM
% J. Matrix Anal. Appl., 16 (1995), pp. 207--217.

% 
% \bibitem{davies higham} P. I. Davies and N. J. Higham, Numerically stable generation of correlation matrices and
% their factors, BIT, 40 (2000), pp. 640--651.
% 

\bibitem{dhillon heath sustik tropp}
Dhillon, Inderjit S.; Heath, Robert W., Jr.; Sustik, M\'aty\'s A.; Tropp, Joel A.  {\it Generalized finite algorithms for constructing Hermitian matrices with prescribed diagonal and spectrum}. SIAM J. Matrix Anal. Appl.  27  (2005), no. 1, 61--71 (electronic). 


\bibitem{duranmatarecht} Dur\'an, C.E., Mata-Lorenzo, L.E. and Recht, L., {\it Metric geometry in homogeneous spaces of the unitary group of a C$^*$-algebra: Part I--minimal curves}, Adv. Math. 184 No. 2 (2004), 342-366.

% \bibitem{haagerup} Haagerup, Uffe. Orthogonal maximal abelian $*$-subalgebras of the $n\times n$ matrices and cyclic $n$-roots.  Operator algebras and quantum field theory (Rome, 1996),  296--322, Int. Press, Cambridge, MA, 1997.

\bibitem{horn} 
 Horn, Alfred. {\it Doubly stochastic matrices and the diagonal of a rotation matrix.} Amer. J. Math. 76, (1954). 620--630.

% 
% \bibitem{leite richa tomei} R. S. Leite, R. W. Richa, and C. Tomei, Geometric proofs of some theorems of Schur-Horn
% type, Linear Algebra Appl., 286 (1999), pp. 149--173.

\bibitem{marshall olkin} Marshall, Albert W.; Olkin, Ingram.  {\it Inequalities: theory of majorization and its applications.} Mathematics in Science and Engineering, 143. Academic Press, Inc. [Harcourt Brace Jovanovich, Publishers], New York-London, 1979.

\bibitem{mirsky} L. Mirsky. {\it Matrices with prescribed characteristic roots and diagonal elements}, J. London Math. Soc. 33, (1958), 14--21.

\bibitem{rieffel} M.A. Rieffel. {\it Leibnitz seminorms and best approximation from C$^*$-subalgebras}, Preprint	arXiv:1008.3733v4 [math.OA].


\bibitem{tadej zyczkowski}
Tadej, W.; Zyczkowski, K. {\it Defect of a unitary matrix}, with an appendix by Wojciech Slomczynski. Linear Algebra Appl. 429 (2008), no. 2-3, 447--481. 	arXiv:math/0702510v2 [math.RA]

%\bibitem{whittaker} Whittaker, E.\ T.\, ``A Treatise on the Analytical Dynamics of Particles and Rigid Bodies'', Cambridge University Press, London 1988.


    \bibitem{zyczkowski kus slomczynski sommers} Zyczkowski, K., Kus,  M.,  Slomczynski, W and Sommers, H.J.,
{\it Random unistochastic matrices},  Journal of Physics A: Mathematical and General 
    vol. 36 (2003), n. 12, 3425-3450.
	arXiv:nlin/0112036v3 [nlin.CD]

% 
% \bibitem{zha zhang} H. Zha and Z. Zhang, A note on constructing a symmetric matrix with specified diagonal
% entries and eigenvalues, BIT, 35 (1995), pp. 448--451.


\end{thebibliography}
\end{document}